\documentclass[a4paper]{article}
\usepackage[english]{babel}
\usepackage[utf8]{inputenc}
\usepackage{paralist,url,verbatim}
\usepackage{amscd}
\usepackage[centering,margin=2.5cm]{geometry}
\usepackage{setspace}
\usepackage{fancyhdr}
\usepackage{mathrsfs}
\usepackage{microtype}
\usepackage{manfnt}
\usepackage{nicefrac}
\usepackage[e]{esvect}
\usepackage{color}
\usepackage[usenames,dvipsnames,svgnames,table]{xcolor}
\definecolor{purple}{HTML}{961C8C}
\usepackage{amsmath,amsthm,amssymb,amsfonts}
\usepackage{graphicx, graphics}
\usepackage[bookmarksopen=true]{hyperref}
\hypersetup{colorlinks=true, citecolor=blue, urlcolor=TealBlue}
\usepackage{epstopdf}
\usepackage[mathlines, pagewise]{lineno}
\usepackage{bbm}
\usepackage[font=small,labelfont=bf]{caption}
\usepackage{subcaption}

\theoremstyle{plain}
\newtheorem{theorem}{\bf Theorem}[section]
\newtheorem{conjecture}[theorem]{Conjecture}

\newtheorem*{theorem*}{Theorem}
\newtheorem*{conjecture*}{Conjecture}
\newtheorem*{problem*}{Problem}
\newtheorem{cor}[theorem]{Corollary}
\newtheorem{lemma}[theorem]{Lemma}
\newcommand\Defn[1]{\emph{\color{RubineRed}#1}}
\newtheorem{prp}[theorem]{Proposition}

\theoremstyle{definition}
\newtheorem{rem}[theorem]{Remark}
\newtheorem{definition}[theorem]{Definition}

\DeclareMathAlphabet{\mathpzc}{OT1}{pzc}{m}{it}
\newenvironment{frcseries}{\fontfamily{frc}\selectfont}{}

\newcommand{\sd}{\operatorname{sd} }

\newcommand{\F}{\operatorname{F} }

\newcommand{\did}{\mathrm{d} }
\newcommand{\RN}{\mathrm{N} }
\newcommand{\TT}{\mathrm{T} }
\newcommand{\cl}{\operatorname{cl} }

\newcommand{\R}{\mathbb{R}}

\newcommand{\intx}{\operatorname{int} }

\newcommand{\diam}{\operatorname{diam} }

\newcommand{\Z}{\mathbb{Z}}
\newcommand{\pp}{\operatorname{p} }

\newcommand{\cm}[1]{}

\newcommand{\Lk}{\mathrm{Lk}}
\newcommand{\St}{\mathrm{St}}

\newcommand{\CAT}{\mathrm{CAT}}

\DeclareFontFamily{OT1}{pzc}{}
\DeclareFontShape{OT1}{pzc}{m}{it}{<-> s * [1.2] pzcmi7t}{}

\begin{document}

\title{The Hirsch conjecture holds for normal flag complexes}
\author{
Karim A. Adiprasito
\thanks{Supported by DFG within the research training group ``Methods for Discrete Structures'' (GRK1408) and by the Romanian NASR, project PN-II-ID-PCE-2011-3-0533.}\\
\small Institut des Hautes \'Etudes Scientifiques\\
\small Le Bois-Marie 35, Route de Chartres\\
\small 91440 Bures-sur-Yvette, France\\
\small \url{adiprasito@math.fu-berlin.de}
\and
\and
Bruno Benedetti \thanks{Supported by the Swedish Research Council, grant ``Triangulerade M{\aa}ngfalder, Knutteori i diskrete Morseteori'', by the KTH Math department, and by the DFG grant ``Discretization in Geometry and Dynamics''.} \\
\small Institut f\" ur Informatik, FU Berlin\\
\small Takustrasse, 9\\
\small 14195 Berlin, Germany\\
\small \url{bruno@zedat.fu-berlin.de}}

\date{\today}
\maketitle 
\begin{abstract}
Using an intuition from metric geometry, we prove that any flag normal simplicial complex satisfies the non-revisiting path conjecture. As a consequence, the diameter of its facet-ridge graph is smaller than the number of vertices minus the dimension, as in the Hirsch conjecture. This proves the Hirsch conjecture for all flag polytopes, and more generally, for all (connected) flag homology manifolds.
\end{abstract}

\section{Introduction}
A natural problem in linear programming is the question how many iteration steps of the simplex method are required in order to solve a linear optimization problem in $d$ variables and given by $n$ linear inequalities. In other words, given an arbitrary polyhedron of dimension $d$ and with $n$ facets, how far away can two vertices possibly be? The distance between vertices is here measured by counting the number of edges one has to walk along, in order to move from one vertex to the other. 

An elegant answer was proposed in the Sixties by Warren Hirsch in a letter to George Dantzig:

\begin{conjecture}[(Unbounded) Hirsch conjecture {\cite[Sec.\ 7.3,\ 7.4]{Dantzig}}]
Let $Q$ denote a $(d+1)$-dimensional polyhedron with $n$ facets. Then the diameter of the $1$-skeleton of $Q$ is $\le n-(d+1)$.
\end{conjecture}

The case of unbounded polyhedra was quickly resolved when a counterexample was given by Klee and Walkup~\cite{KleeWalkup}. It remained to treat the case of bounded polyhedra (that is, polytopes), the \emph{bounded Hirsch conjecture}. We state the conjecture in a form dual to the classical formulation.

\begin{conjecture}[(Bounded) Hirsch conjecture {\cite{KleeWalkup}}]\label{cj:BHC}
The diameter of the facet-ridge graph of any $(d+1)$-polytope on $n$ vertices is $\le n-(d+1)$.
\end{conjecture}
 
An equivalent conjecture, the \emph{$W_v$-conjecture}, or \emph{non-revisiting path conjecture}, was introduced in the Sixties by Klee and Wolfe, cf.~\cite{Klee}.

\begin{conjecture}[Non-revisiting path conjecture, or $W_v$-conjecture]\label{cj:BWv}
For any two facets of a simplicial polytope $R$ there exists a non-revisiting path connecting them.
\end{conjecture}

Here, a path of facets $\Gamma$ (represented by a map from an interval $I \in \mathbb{Z}$ to the facets of $\Sigma$) in a simplicial complex $\Sigma$ is \emph{non-revisiting} if the preimage of $\Gamma$, restricted to the star of any vertex of $\Sigma$, is an interval itself (cf.\ Section \ref{ssc:setup}). The reason why the $W_v$-conjecture implies the Hirsch conjecture is simple: Any non-revisiting path can be at most $n-(d+1)$ steps long. Here is why: At the beginning of the path we are in some $d$-face $X_0$, which has (at least) $d+1$ vertices. Next, we step into a new facet $X_1$, and we see a new vertex. From that moment on at each step we have to see a new vertex, otherwise the path would be revisiting. Since there are $n$ vertices in total, after $n-(d+1)$ steps we have seen all vertices already! 

Conversely, Klee and Kleinschmidt \cite{KleeKleinschmidt} showed that if there is a polytope $R$ which violates the $W_v$-conjecture, then from $R$ one can construct a (possibly different) polytope $P$ that violates the Hirsch conjecture.

Conjectures~\ref{cj:BHC} and~\ref{cj:BWv} have been disproved recently by Santos \cite{Santos}. So the bound $n-(d+1)$ for the diameter is not correct. Little do we know about how the correct bound should look like. At the moment, we do not know whether a \emph{linear} or even a \emph{polynomial} upper bound exist. Some of the best upper bounds known so far are the bounds 
$2^{d-1}n$ by Larman \cite{Larman} (compare also \cite{Barnette}), and the bound $n^{\log (d+1)+1}$ by Kalai \cite{Kalaidiameter, KaKl}. These bounds apply more generally to the class of \emph{normal} $d$-complexes, i.e., complexes where all links of faces of codimension $\ge 2$  are connected. (This is a common setting for the study of abstractions of the Hirsch conjecture, compare also Eisenbrand et al \cite{EHRR}.)

In this paper, we confirm the validity of the Hirsch conjecture for flag polytopes and more generally flag and normal complexes. All polytope boundaries, all spheres, all triangulated manifolds, and even all Cohen--Macaulay complexes are normal. 

\begin{theorem}\label{THM:HIRSCHA}
Let $C$ be any flag normal $d$-complex with $n$ vertices. Between any two facets of $C$ there is a non-revisiting path. Hence, the dual graph of $C$ has diameter $\le n-(d+1)$.
\end{theorem}

We provide two proofs of Theorem \ref{THM:HIRSCHA}: a geometric proof (Section~\ref{sec:geomproof}), which follows from a result by Gromov on spaces of curvature bounded above, and a combinatorial proof (Section~\ref{sec:combproof}), which is more elementary, but also less intuitive.

Here is a sketch of and intuition for the geometric proof. Any simplicial complex can be turned into a metric length space by assigning the same length $\nicefrac{\pi}{2}$ to all edges of $C$, and interpreting all $k$-faces of $C$ to be equilateral simplices in the unit sphere $S^k\in\R^{k+1}$. Gromov \cite{GromovHG} revealed an interesting connection between geometric properties of this ``right-angled" metric and flag complexes:

\begin{quote}\em 
 Let $C$ be any \emph{flag} simplicial complex. When we endow $C$ with the right-angled metric, the star of every vertex of $C$ is geodesically convex.
 \end{quote}
 
\begin{figure}[htbf]
	\centering
  \includegraphics[width=.7\linewidth]{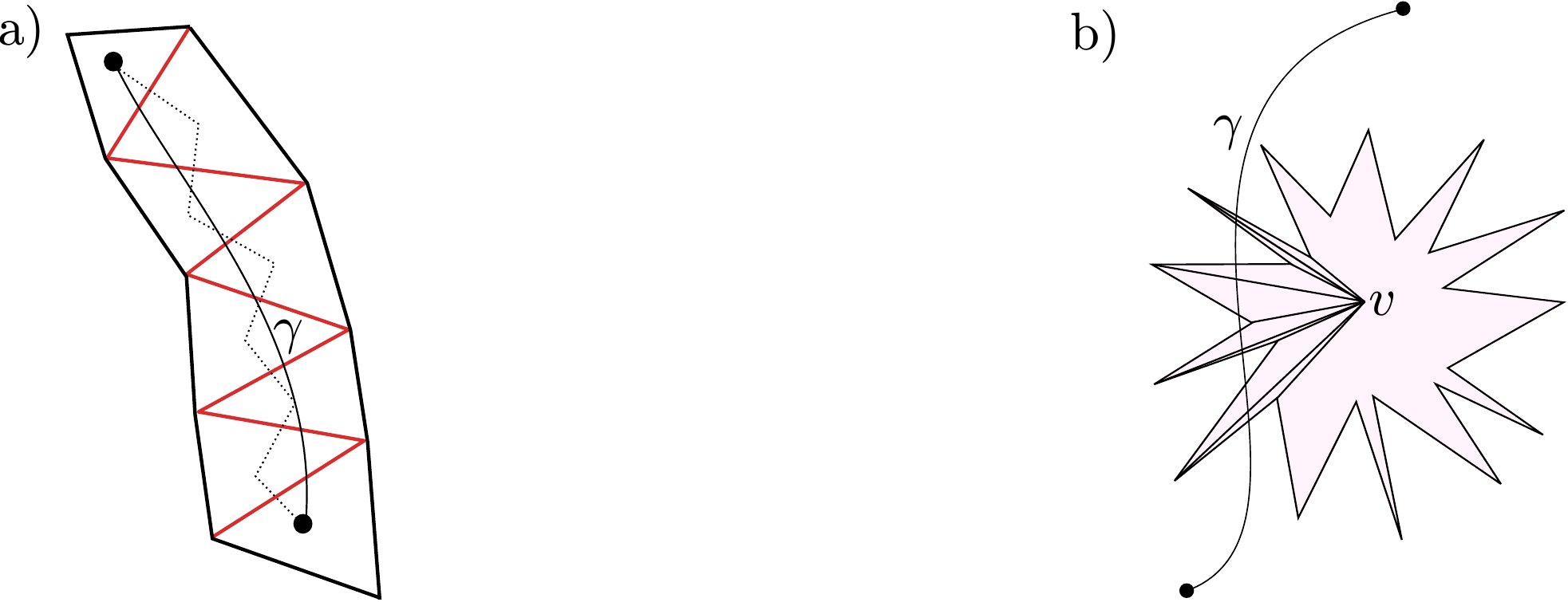}
 	\caption{a): If all vertex stars are convex, any segment $\gamma$ intersect any vertex star in a convex segment (and therefore does not reenter the star of a vertex it has previously left). By following the segment, we obtain our desired facet path.  b): A non-example. If some vertex star is not convex, some segment  $\gamma$  revisits it multiple times.}
	\label{fig:hirsch}
\end{figure}
 
 \noindent Say we have a flag normal complex and we want to find a non-revisiting path. Our idea is to endow it with the right-angled metric, and then `follow' the segments, that is, the shortest geodesics inside the metric space.  In fact, the intersection of any segment with an open convex set is obviously a segment (or the empty set). In particular, any segment intersects the interior of any vertex star in a connected set (possibly empty). In other words, no segment revisits a vertex star it has previously left. If we approximate a segment $\gamma$ with the dual path formed by the $d$-faces crossed by $\gamma$, the path we obtain is non-revisiting and we are done.

While the `flag' assumption is needed for the convexity of vertex stars, one might wonder whether the `normal' assumption is at all needed in the argument above. The truth is that we have hidden a minor technical difficulty under the carpet. Namely, a segment might go from a $d$-face $X$ to  a $d$-face $Y$ by passing through a face $\sigma$ of dimension $\le d-2$. Even if they share a vertex, $X$ and $Y$ are not (necessarily) adjacent in the dual graph; so if $C$ is not normal, it is not clear how to find a non-revisiting dual path from $X$ to $Y$ inside the star of $\sigma$. The natural way to ``bridge'' between $X$ and $Y$, is to consider the link of $\sigma$ and use induction. For this we need $C$ to be normal.

\subsection{Consequences}\label{sec:corH}

Theorem~\ref{THM:HIRSCHA} has several interesting consequences. Recall that a simplicial complex is a \Defn{triangulated manifold} if the union of its faces, as topological space, is homeomorphic to a manifold. 

\begin{cor}\label{cor:HirschB}
All flag triangulations of connected manifolds satisfy the non-revisiting path property, and in particular the Hirsch diameter bound.
\end{cor}

Recall that a simplicial polytope is \Defn{flag} if its boundary complex is a flag complex. Corollary \ref{cor:HirschB} specializes to this class as follows:

\begin{cor}\label{cor:HirschC}
Every flag polytope satisfies the non-revisiting path property, and in particular the Hirsch diameter bound.
\end{cor}

\begin{rem}\label{rem:prb}
By a result of Provan and Billera~\cite{PB}, every \Defn{vertex-decomposable} simplicial complex satisfies the Hirsch diameter bound. As a corollary, they obtain the following famous result:

\begin{theorem}[Provan {\&} Billera {\cite[Cor.\ 3.3.4.]{PB}}]
Let $C$ be any shellable simplicial $d$-complex. Then the derived subdivision $\operatorname{sd} C$ of $C$ satisfies the Hirsch diameter bound. In particular, if $C$ is the boundary complex of any polytope, then $\sd C$ satisfies the Hirsch diameter bound.
\end{theorem}
The derived subdivision of an arbitrary triangulated manifold, however, is not vertex-decomposable in general. The reasons are two: There are topological obstructions (all vertex-decomposable manifolds are spheres or balls) as well as combinatorial obstructions (some spheres have non-vertex-decomposable derived subdivisions, cf.~\cite{HZ, BZ}). That said, the derived subdivision of any simplicial complex is flag. So, by Corollary~\ref{cor:HirschB}, we have the following:

\begin{cor}
The derived subdivision of any triangulation of any connected manifold satisfies the Hirsch diameter bound.
\end{cor}
 
\end{rem}

\subsection{Set-up}\label{ssc:setup}

Recall that an (abstract) simplicial complex is \Defn{pure} if all its inclusion-maximal faces (the \Defn{facets}) have the same dimension. If $C$ is an abstract simplicial complex on $n$ vertices, any subset of $\{1,\ldots, n\}$ not in $C$ is a \Defn{non-face}. 

\begin{definition}[Flag complexes]
A simplicial complex $C$ is \Defn{flag} if every inclusion minimal non-face is a
$2$-element set (that is, an edge).
\end{definition}

\begin{definition}[Diameter of (the dual graph of) a complex]
If $C$ is a pure simplicial $d$-complex on $n$ vertices, the \Defn{dual graph} or \Defn{facet-ridge graph} of $C$, denoted by $G^\ast(C)$, is constructed as follows. The set of vertices of $G^\ast(C)$ consists of the facets of $C$; we connect two vertices by an edge if the corresponding facets have a $(d-1)$-face in common. We define $\diam(C)$ as the diameter of the graph $G^\ast(C)$.  We say that $C$ satisfies the \Defn{Hirsch diameter bound} if $\diam(C)\le n-(d+1).$
\end{definition}

Recall that if $\sigma$ is a face of an abstract simplicial complex $C$, the \Defn{star} $\St(\sigma,C)$ of $\sigma$ in $C$ is the collection of faces $\tau$ of $C$ with the property that $\tau\cup \sigma\in C$; the \Defn{link} $\Lk(\sigma,C)$ of $\sigma$ in $C$ is the collection of faces $\tau$ of $C$ such that $\tau \cap \sigma=\varnothing$, but $\tau \cup \sigma \in C$. 
If $\sigma$, $\tau$ are two faces of a simplicial complex $C$ that lie in a common face, then $\sigma\ast\tau$, the \Defn{join} of $\sigma$ and $\tau$, denotes the minimal face of $C$ containing them both. With this, the link of a face $\sigma$ in $C$ is combinatorially isomorphic to the complex
\[\widetilde{\Lk}(\sigma,C):=\{\tau \in C:\sigma\ast \tau \in C, \sigma\cap \tau=\varnothing\}\]
the \Defn{combinatorial link} via the map
\begin{align*}
\widetilde{\Lk}(\sigma,C)\ &\ \longrightarrow\ \ {\Lk}(\sigma,C)\\
\tau\ &\  \longmapsto\ \ \Lk(\sigma,\sigma\ast \tau).
\end{align*}
We shall therefore identify elements of link and combinatorial link.

\begin{definition}[Normal complexes]
Let $C$ be a pure simplicial $d$-complex. A pure simplicial complex $C$ is \Defn{normal} if for every face $\sigma$ of $C$ (including the empty face), $G^\ast(\St(\sigma,C))$ is connected. \end{definition}

For the next definition, we use the notation $\F_k (C)$ to denote the set of faces of $C$ of dimension $k$ (or equivalently, of cardinality $k+1$). By an \Defn{interval} in $\mathbb{Z}$ 
we mean a set of the type $[a,b]:=\{x\in \mathbb{Z}: a\le x\le b\}$. 

\begin{definition}[Curves, facet paths and vertex paths]
If $X$ is a metric space and $I$ is an interval in~$\R$, an immersion $\gamma:I\mapsto X$ is a \Defn{curve}. If $C$ is a pure simplicial $d$-complex, and $I$ is an interval in~$\Z$, then a \Defn{facet path} is a map $\Gamma$ from $I$ to $\F_d(C)$ such that for every two consecutive elements $i$, $i+1$ of~$I$, we have that $\Gamma(i)\cap \Gamma(i+1)$ has dimension $d-1$. A \Defn{vertex path} in $C$ is a map $\gamma$ from $I$ to $\F_0(C)$ such that for every two consecutive elements $i$, $i+1$ of~$I$, the vertices $\gamma(i)$ and $\gamma(i+1)$ are joined by an edge. 
\end{definition}

All curves and paths are considered with their natural order from the startpoint (the image of $\min I$) to the endpoint (the image of $\max I$). For example, the \Defn{last facet} of a facet path $\Gamma$ in a subcomplex $S$ of $C$ is the image of the maximal $z\in I$ such that $\gamma(z)\in S$. As common in the literature, we will not strictly differentiate between a curve (or path) and its image; for instance, we will write $\gamma\subset S$ to denote the fact that the image of a curve $\gamma$ lies in a set $S$.

If $\gamma$ and $\delta$ are two curves in any metric space such that the endpoint of $\gamma$ coincides with the starting point of $\delta$, we use the notation $\gamma \cdot \delta$ to denote their \Defn{concatenation} or \Defn{product} (cf.~\cite[Sec.~2.1.1.]{BuragoBuragoIvanov}). Analogously, if the last facet of a facet path $\Gamma$ and the first facet of a facet path $\Delta$ coincide, we can \Defn{concatenate} $\Gamma$ and $\Delta$ to form a facet path $\Gamma\cdot \varDelta$. Concatenations of more than two paths are represented using the symbol~$\prod$. 

If $i$ and $j$ are elements in the domain of a facet path $\Gamma$, then $\Gamma_{[i,j]}$ is the restriction of $\Gamma$ to the interval $[i,j]$ in $\mathbb{Z}$. If a facet path $\Gamma$ is obtained from a facet path $E$ by restriction to some interval, then $\Gamma$ is a \Defn{subpath} of $E$, and we write $\Gamma \subset E$. Two facet paths coincide up to \Defn{reparametrization} if they coincide up to an order-preserving bijection of their respective domains.

\begin{definition}[$W_v$-property, cf.~\cite{Klee}]
Let $C$ be a pure simplicial complex. The facet path $\Gamma$ is \Defn{non-revisiting} if for every pair $i, j$ in the domain of $\Gamma$ such that $\Gamma(i)$ and $\Gamma(j)$ lie in $\St(v,C)$ for some vertex $v\in C$, the subpath $\Gamma_{[i,j]}$ of $\Gamma$ lies in $\St(v,C)$. Equivalently, $\Gamma$ is non-revisiting if for every vertex $v$ of $C$, the preimage $\Gamma^{-1}(\St(v,C))$ is an interval in $\mathbb{Z}$. We say that $C$ satisfies the \Defn{non-revisiting path property}, or \Defn{$W_v$-property}, if for every pair of facets of $C$, there exists a non-revisiting facet path connecting the two.
\end{definition}

\begin{lemma}[cf.~\cite{KleeKleinschmidt}]
Any pure simplicial complex that satisfies the $W_v$-property satisfies the Hirsch diameter bound.
\end{lemma}

Finally, if $M$ is a metric space with metric $\did: M\times M \mapsto \R$, then the \Defn{distance between two subsets} $A, B$ of $X$ is defined as $\did(A,B):=\inf \{\did(a,b): a\in A, b\in B\}.$ 
 
\section{The geometric proof}\label{sec:geomproof}

In this section, we give a geometric proof of Theorem~\ref{THM:HIRSCHA}. We need some modest background from the theory of spaces of curvature bounded above, which we review here. For a more detailed introduction, we refer the reader to the textbook by Burago--Burago--Ivanov~\cite{BuragoBuragoIvanov}.

\subsubsection*{$\CAT(1)$ spaces and convex subsets} 

A metric space $M$ with metric $\did: M\times M \mapsto \R$ is a \Defn{length space} if for every pair of points $a$ and $b$ in the same connected component of $M$, the value of $\did(a,b)$ is also the minimum of the lengths of all rectifiable curves from $a$ to~$b$.
A curve that attains the distance $\did(a,b)$ is denoted by $[a,b]$, and is a \Defn{segment} connecting $a$ and $b$. A \Defn{geodesic} $\gamma:I\mapsto M$ is a curve that is locally a segment, that is, every point in $I$ has an open neighborhood $J$ such that $\gamma$, restricted to $\cl(J)$, is a segment. A \Defn{geodesic triangle} $[a,b,c]$ in $M$ is given by three vertices $a,b,c$ connected by some three segments  $[a,b],\, [b,c]$ and $[a,c]$, each of length $<\pi$. 

A \Defn{comparison triangle} for a geodesic triangle $[a,b,c]$ in $M$ is a geodesic triangle $[\bar{a},\bar{b},\bar{c}]$ in $S^2$ such that $\did(\bar{a},\bar{b})=\did(a,b)$, $\did(\bar{a},\bar{c})=\did(a,c)$ and $\did(\bar{b},\bar{c})=\did(b,c)$. The space $M$ is a \Defn{$\CAT(1)$ space} if it is a length space in which the following condition is satisfied:

\smallskip
\noindent {\textsc{Triangle condition}:} \emph{For each geodesic triangle $[a,b,c]$ inside $M$ and for any point $d$ in the relative interior of $[a, b]$, one has $\did(c,d)\leq \did(\bar{c},\bar{d})$, where $[\bar{a},\bar{b},\bar{c}]$ is any comparison triangle for $[a,b,c]$ and $\bar{d}$ is the unique point on $[\bar{a},\bar{b}]$ with $\did(a,d) = \did(\bar{a},\bar{d})$.}
\smallskip

Let $A$ be any subset of a length space $M$. The set $A$ is \Defn{convex} if any two points of $A$ are connected by a segment that lies in $A$. The set $A$ is \Defn{locally convex} if every point in $A$ has an open neighborhood $U$ such that $U\cap A$ is convex. The following classical observation relates convexity and local convexity in $\CAT(1)$ spaces.

\begin{prp}[cf.~\cite{Tietze},~\cite{Nakajima}, {\cite[Thm.\ 8.3.3]{Papa}}, \cite{BuxWitzel}]\label{prp:hdct} Let $M$ denote a compact $\CAT(1)$ length space. Let $A$ be any locally convex subset of $M$ such that any two points in $A$ are connected by a rectifiable curve in $A$ of length $\le \pi$. Then $A$ is convex.
\end{prp}

\subsubsection*{Right-angled simplices and convex vertex-stars}
\newcommand\geo{{\mathrm{geo}}}
\newcommand\glue{\mathrm{gl}}

For us, a \Defn{geometric (spherical) simplex} of dimension $d$, or geometric \Defn{d-simplex}, is the convex hull of $d+1$ points in general position in $S^d$. A geometric simplex $\Delta$ is \Defn{right-angled} if all dihedral angles of $\Delta$ are equal to $\nicefrac{\pi}{2}$. Equivalently, $\Delta$ is right-angled if it is regular and of diameter $\nicefrac{\pi}{2}$. By convention, every $0$-simplex is \Defn{right-angled} as well. 

Naturally, if $C$ is a simplicial complex, we can assign to every face $\sigma$ in $C$ a right-angled geometric simplex $\sigma_\geo$, and subsequently glue the geometric simplices along faces using the combinatorial information given by $C$. Since right-angled simplices of the same dimension are isometric, we can choose the gluing maps to be isometries. We say the resulting object $C_\geo$ is an \Defn{intrinsic simplicial complex}, and the distance between two points $a$, $b$ in $C_\geo$ is given by the minimum over the length of all rectifiable curves connecting $a$ to $b$; this is the natural \Defn{intrinsic length metric} on~$C_\geo$. 

For a more detailed introduction to the intrinsic geometry of simplicial complexes, we refer the reader to \cite[\S 3.2]{BuragoBuragoIvanov}, \cite{Charney} and \cite[Sec.\ 2.1]{DM-NP}. For the rest of this section we consider every simplicial complex $C$ to be endowed with its intrinsic length metric $\did$.

If $C$ is any intrinsic simplicial complex, and $\sigma$ is any face of $C$, then the link $\Lk(\sigma, C)$ has a natural geometric structure itself: If $p$ is any interior point of $\sigma$, then $\RN^1_{(p,\sigma)} C$ is the subset of unit length elements of the tangent space $\TT_{(p,\sigma)} C$ that are orthogonal to $\sigma$. The space $\RN^1_{(p,\sigma)} C$ is naturally subdivided into (right-angled) simplices itself: if $\tau$ is any face of $C$ containing $\sigma$, then $\RN^1_{(p,\sigma)} \tau$, the subset of elements $\RN^1_{(p,\sigma)} C$ ``pointing towards'' $\tau$, is isometric to a simplex in some sphere $S^d$. The collection $\Lk_p(\sigma, C)$ of spherical simplices obtained this way is a intrinsic simplicial complex that is combinatorially equivalent to the link $\Lk(\sigma, C)$ of $C$ at $\sigma$; in this section, we shall identify the two. This is well defined: up to isometry, 
$\Lk_p(\sigma, C)$ does not depend on the choice of $p$. For details, see Charney \cite{Charney} or \cite[Sec.\ 2.2]{DM-NP}.

 With this notion, Proposition~\ref{prp:hdct} gives the following:

\begin{cor}\label{cor:hdct}
Let $C$ be a pure simplicial $d$-complex such that each face of $C$ is right-angled and $C$ is a $\CAT(1)$ metric space. Then $\St(v,C)$ is convex in $C$ for every vertex $v$ of $C$.
\end{cor}

\begin{proof}
The proof is by induction on $d$; the case $d=0$ is trivial. Assume now $d\ge 1$. For every vertex $w\in C$, the simplicial complex $\Lk (w,C)$ is a $\CAT(1)$ complex (cf.~\cite[Thm.\ 4.2.A]{GromovHG}) all whose faces are right-angled. Thus, $\St(v,C)$ is locally convex since for every $w\in \St(v,C), w\neq v$, we have that $\Lk(w,\St(v,C))=\St(v,\Lk(w,C))$ is convex in $\Lk (w,C)$ by inductive assumption. Furthermore, since every face of $C$ is right-angled, every point in $\St(v,C)$ can be connected to $v$ by a segment in $\St(v,C)$ of length $\le \nicefrac{\pi}{2}$. Application of Proposition~\ref{prp:hdct} finishes the proof.
\end{proof}

\subsection*{Geometric proof of Theorem~\ref{THM:HIRSCHA}.}

\begin{lemma} \label{lem:Hirsch}
Let $C$ be a normal simplicial $d$-complex such that each simplex of $C$ is right-angled and $C$ is a $\CAT(1)$ metric space. Let $X$ be any facet of $C$, and let $\mathcal{Y}$ be any finite set of points in $C$. Then, there exists a non-revisiting facet path $\Gamma$ from the facet $X$ of $C$ to some facet of $C$ containing a point of $\mathcal{Y}$.
\end{lemma}

\begin{proof}
The proof, as well as the construction of the desired facet path, is by induction on the dimension $d$ of $C$. The case $d=0$ is easy: If $X$ consists of an element of $\mathcal{Y}$, the path is trivial of length $0$. If not, the desired facet path is given by $\Gamma:\{0,1\}\mapsto C$, with $\Gamma(0):=X$ and $\Gamma(1):=Y$, where $Y$ is any facet of $C$ consisting of an element of $\mathcal{Y}$. We proceed by induction on $d$, assuming that $d\ge 1$. 

\smallskip

\emph{Some preliminaries}: If $\alpha$ is any point in $C$, let us denote by $\sigma_\alpha$ the minimal face of $C$ containing $\alpha$. If $\omega$ is any second point in $C$, let $\operatorname{S}_{\alpha}^{\omega}$ denote the set of segments from $\alpha$ to $\omega$. For an element $\gamma\in\operatorname{S}_{\alpha}^{\omega}$ with $\TT_\alpha^1 \gamma\notin \TT_\alpha^1 \sigma_\alpha$, \Defn{the tangent direction of $\gamma$ in $\Lk_\alpha(\sigma_{\alpha},C)=\Lk(\sigma_{\alpha},C)$ at $\alpha$} is defined as the barycenter of $\Lk(\sigma_{\alpha},\tau)$, where $\tau$ is the minimal face of $C$ that contains $\sigma_\alpha$ and such that $\TT_\alpha^1 \gamma\in \TT_\alpha^1 \tau$.
Define $\operatorname{T}_{\alpha}^{\omega}$ to be the union of tangent directions in $\Lk (\sigma_{\alpha},C)$ at $\alpha$ over all segments $\gamma\in\operatorname{S}_{\alpha}^{\omega}$. Finally, set
\[\operatorname{S}_{\alpha}^{\mathcal{\varOmega}}:=\bigcup_{\omega\in \mathcal{\varOmega}} \operatorname{S}_{\alpha}^{\omega}\ \ \text{and}\ \ \operatorname{T}_{\alpha}^{\mathcal{\varOmega}}:=\bigcup_{\omega\in \mathcal{\varOmega}} \operatorname{T}_{\alpha}^{\omega}.\]
for any collection $\mathcal{\varOmega}$ of points in $C$. Clearly, $\operatorname{T}_{\alpha}^{\mathcal{\varOmega}}$ is finite.

\smallskip

Returning to the proof, let $x_0$ denote any point of $X_0:=X$ minimizing the distance to the set $\mathcal{Y}_0:=\mathcal{Y}$. Set $i:=0$. The construction process for the desired facet path goes as follows: 

\smallskip

\noindent {\bf Construction procedure.} If ${X_i}$ intersects $\mathcal{Y}$, set $\ell:=i$ and stop the procedure. If not, consider the face $\sigma_i:=\sigma_{x_i}$ of $C$ containing $x_i$ in its relative interior $\sigma_i\setminus \partial \sigma_i$. The simplicial complex $\Lk (\sigma_{i},C)$ is a normal $\CAT(1)$ complex (cf.~\cite[Thm.\ 4.2.A]{GromovHG}) all whose faces are right-angled. Now, we use the construction technique for dimension $d-\dim \sigma_i - 1\le d-1$ to find a (non-revisiting) facet path $\Gamma'_{X'_iX'_{i+1}}$ in $\Lk(\sigma_{i},C)$ from $X'_i:=\Lk(\sigma_{i},X_i)$ to some facet $X'_{i+1}$ of $\Lk(\sigma_{i},C)$ that intersects $\operatorname{T}_{x_i}^{\mathcal{Y}_i}$. We may assume that $\Gamma'_{X'_iX'_{i+1}}$ intersects  $\operatorname{T}_{x_i}^{\mathcal{Y}_i}$ only in the last facet $X'_{i+1}$. Lift the facet path $\Gamma'_{X'_iX'_{i+1}}$ in $\Lk(\sigma_{i},C)$ to a facet path $\Gamma_{X_iX_{i+1}}$ in $C$ from $X_i$ to $X_{i+1}:=\sigma_{i}\ast X'_{i+1}$ by join with $\sigma_i$, i.e.\ define \[\Gamma_{X_iX_{i+1}}:=\sigma_i\ast \Gamma'_{X_i'X_{i+1}'}.\] 
Let $\gamma_i$ be any element of $\operatorname{S}_{x_i}^{\mathcal{Y}_i}$ whose tangent direction in $\Lk(\sigma_{i},C)$ at $x_i$ lies in $X'_{i+1}$, let $\overline{\gamma}_i$ denote the restriction of $\gamma_i$ to $X_{i+1}$, and let $x_{i+1}$ be the last point of $\gamma_i$ in $X_{i+1}$. Finally, let $\mathcal{Y}_{i+1}$ denote the subset of points $y$ of $\mathcal{Y}_i$ with
\begin{equation}\tag{$\ast$} \label{eq:xx}
\did(y,x_i)=\did(y,x_{i+1})+\did(x_i,x_{i+1}).
\end{equation}
Now, increase $i$ by one, and repeat the construction procedure from the start.

\smallskip

Define the facet path \[\Gamma:=\prod_{i\in (0,\, \cdots,\ell-1)} \Gamma_{X_{i}X_{i+1}}.\] Associated to $\Gamma$, define the curve \[\gamma=\prod_{i\in(0,\, \cdots,\ell-1)} \overline{\gamma}_i \] from $x$ to some element $y$ of $\mathcal{Y}$, the \Defn{necklace} of $\Gamma$, and define the \Defn{pearls} of $\Gamma$ to be the faces $\sigma_i$. Finally, we denote by $\chi_i$, $0\le i\le \ell$, the element in the domain of $\Gamma$ corresponding to $X_i$; with this, the facet paths $\Gamma_{X_{i}X_{i+1}}$ coincide up to reparametrization with the subpaths $\Gamma_{[\chi_{i},\chi_{i+1}]}$ of $\Gamma$ for each $i$. 
For any element $a\neq \chi_\ell$ in the domain of $\Gamma$, let $i$ be chosen so that $a\in [\chi_{i},\chi_{i+1}-1]$. We say that $a$ is \Defn{associated to the pearl $\sigma_i$} of $\Gamma$. By convention, \Defn{$\chi_\ell$ is associated to the pearl~$\sigma_{\ell-1}$}.

By Equation \eqref{eq:xx}, $\gamma$ is a segment. Thus, by Corollary~\ref{cor:hdct}, if $v$ is any vertex of $C$, then $\gamma$ intersects $\intx \St(v,C)$ in a connected component. We will see that this fact extends to the combinatorial setting. First, we make the following claim.

\smallskip 
\noindent \emph{Let $a$ denote any element in the domain of $\Gamma$, and let $v$ be any vertex of $\Gamma(a)$. Let $\hat{x}$ denote the last point of $\gamma$ in $\St(v,C)$, and assume that $\hat{x}$ is not in $\Gamma(a)$. Let $\sigma_i$ be the pearl associated to $a$. Then $\Gamma_{[a,\chi_{i+1}]}$ lies in $\St(v,C)$. In particular, $X_{i+1}$ lies in $\St(v,C)$.}
\smallskip

\noindent To prove the claim, we need only apply an easy induction on the dimension:
\begin{compactitem}[$\circ$]
\item If $v$ is a vertex of the pearl $\sigma_i$, this follows directly from construction of $\Gamma$. Now, if $d=1$, then one of the vertices of $\Gamma(a)$, the last one encountered by $\gamma$, must be the pearl associated to $a$. But since $\hat{x}\notin \Gamma(a)$ comes after $v$ along $\gamma$ in $\St(v,C)$, we therefore have that $v$ must be the pearl associated to $a$, which in particular proves the case $d=1$. 
\item If $v$ is not in $\sigma_i$, we consider the facet path $\Gamma':=\Lk(\sigma_i,\Gamma_{[\chi_i,\chi_{i+1}]})$ in $\Lk(\sigma_i,C)$. The point $\hat{x}$ lies in $\St(v,C)$, which is convex in $C$ by Corollary \ref{cor:hdct}. Thus, the construction of $\gamma$ and $\Gamma$ implies that the restriction of $\gamma$ to the interval $[\gamma^{-1}(x_i),\gamma^{-1}(\hat{x})]$ lies in $\St(v,C)$: indeed, if $\did(x_i,\hat{x})<\pi$, then $x_i$ and $\hat{x}$ are connected by a unique segment in $C$, thus, this segment must lie in $\St(v,C)$. If $\did(x_i,\hat{x})\geq \pi$, then connecting $x_i$ to $v$ and $v$ to $\hat{x}$ by segments gives a segment from $x_i$ to $\hat{x}$; thus, $\Gamma(a)$ must contain $\sigma_{i+1}$ by construction of $\Gamma$, which contradicts the assumption that $a$ was associated to $\sigma_i$.

In particular, since $[\gamma^{-1}(x_i),\gamma^{-1}(\hat{x})]$ lies in $\St(v,C)$, the tangent direction of $\overline{\gamma}_i$ in $\Lk(\sigma_i,C)$ at $x_i$ is a point in $\St(v,\Lk(\sigma_i,C))$, where $v$. However, since $\Gamma(a)$ does not contain $\hat{x}$, this tangent direction does not lie in~$\Gamma'(a)$. Hence, the path $\Gamma'_{[a,\chi_{i+1}]}=\Lk(\sigma_i,\Gamma_{[a,\chi_{i+1}]})$ is contained in $\St(v,\Lk(\sigma_i,C))$ by induction assumption. We obtain \[\Gamma_{[a,\chi_{i+1}]}=\sigma_i\ast \Gamma'_{[a,\chi_{i+1}]}\subset\sigma\ast \St(v,\Lk(\sigma_i,C))\subset\St(v,C).\]
\end{compactitem}
We can now use induction on $d$ to conclude that $\Gamma$ is non-revisiting. The case $d=0$ is trivial, assume therefore $d\ge 1$. Let $v$ be any vertex of $C$, and let $a, b$ denote elements in the domain of $\Gamma$ with $\Gamma(a),\Gamma(b) \in \St(v,C)$ such that $a\le b$. We have to prove that the image of $\Gamma_{[a,b]}$ lies in $\St(v,C)$. Let $j$, $j\ge i$, be chosen such that $\sigma_j$ is the pearl associated with~$b$ and $\sigma_i$ is the pearl associated to $a$. There are two cases to consider
\begin{compactitem}[$\circ$]
\item {\bf If $i=j$:} By induction assumption, the facet path $\Gamma'_{X'_iX'_{i+1}}$ (as defined above) is non-revisiting. Thus, the facet path $\Gamma_{[\chi_i,\chi_{i+1}]}$, which coincides with $\Gamma_{X_iX_{i+1}}=\sigma_i\ast \Gamma'_{X'_iX'_{i+1}}$ up to reparametrization, is non-revisiting. Since $\Gamma_{[a,b]}$ is a subpath of $\Gamma_{[\chi_i,\chi_{i+1}]}$, this finishes the proof of this case.
\item {\bf If $i<j$:} The claim proves that $\Gamma_{[a,\chi_{i+1}]}$ lies in $\St(v,C)$ and that for every $k$, $i< k < j$, $\Gamma_{[\chi_k,\chi_{k+1}]}$ lies in $\St(v,C)$. Thus, \[\Gamma_{[a,\chi_{j}]}= \Gamma_{[a,\chi_{i+1}]} \cdot \Big(\prod_{k,\ i< k< j} \Gamma_{[\chi_k,\chi_{k+1}]}\Big)\subset \St(v,C).\]
Since $\Gamma_{[a,b]}=\Gamma_{[a,\chi_{j}]}\cdot\Gamma_{[\chi_{j},b]}$, it only remains to prove that  $\Gamma_{[\chi_{j},b]}\subset \St(v,C)$; this was proven in the previous case. \qedhere
\end{compactitem}
\end{proof}

\begin{cor} \label{cor:hirsch}
Let $C$ be a normal simplicial $d$-complex such that each simplex of $C$ is right-angled and $C$ is a $\CAT(1)$ metric space. Then $C$ satisfies the non-revisiting path property.  
\end{cor}

\begin{proof}
If $X$ and $Y$ are any two facets of $C$, apply Lemma 
\ref{lem:Hirsch} to the facet $X$ and the set $\mathcal{Y}=\{y\}$, where $y$ is any interior point of $Y$.
\end{proof}

We can give the first proof of Theorem~\ref{THM:HIRSCHA}.
\begin{proof}[\textbf{Proof of Theorem~\ref{THM:HIRSCHA}}]
We turn $C$ canonically into a length space,
by endowing  every face with the metric of a regular spherical simplex with dihedral angles $\nicefrac{\pi}{2}$. By Gromov's Criterion~\cite[Sec.\ 4.2.E]{GromovHG}, the resulting metric space is $\CAT(1)$, because $C$ is flag. By construction, every simplex of $C'$ is right-angled, so we can apply Corollary~\ref{cor:hirsch} and conclude that $C$ satisfies the non-revisiting path property.
\end{proof}

\section{The combinatorial proof}\label{sec:combproof}

In this section, we give a purely combinatorial proof of Theorem~\ref{THM:HIRSCHA}. The proof is articulated into two parts: First we construct a facet path between any pair of facets of $C$, the so called combinatorial segment, and then we prove that the constructed path satisfies the non-revisiting path property.

Let $\did(x,y)$ denote the (combinatorial) distance between two vertices $x,\, y$ in the $1$-skeleton of the simplicial complex $C$. If $\mathcal{Y}$ is a subset of $\F_0(C)$, let $\pp(x,\mathcal{Y})$ denote the elements of $\mathcal{Y}$ that realize the distance $\did(x,\mathcal{Y})$, and let ${\operatorname{T}}_x^{\mathcal{Y}}$ denote the set of vertices $y$ in $\Lk(x,C)$ with the property that \[\did(y,\mathcal{Y})+1=\did(x,\mathcal{Y}).\]

\subsection*{Construction of a combinatorial segment}\label{ssc:cseg}
\enlargethispage{2mm}
\noindent{\bf Part 1: From any facet $X$ to any vertex set $\mathcal{Y}$.}

\smallskip 
We construct a facet path from a facet $X$ of $C$ to a subset $\mathcal{Y}$ of $\F_0(C)$, i.e.\ a facet path from $X$ to a facet intersecting $\mathcal{Y}$, with the property that $\mathcal{Y}$ is intersected by the path $\Gamma$ only in the last facet of the path. 

If $C$ is $0$-dimensional, and $X$ consists of an element of $\mathcal{Y}$, the path is trivial of length $0$. Else, the desired facet path is given by $\Gamma:\{0,1\}\mapsto C$, where $\Gamma(0):=X$ and $\Gamma(1):=Y$, which is any facet consisting of an element of $\mathcal{Y}$ . 

If $C$ is of a dimension $d$ larger than~$0$, set $X_0:=X$, let $x_0$ be any vertex of $X_0$ that minimizes the distance $\did$ to $\mathcal{Y}$, set $\mathcal{Y}_0:=\pp({x_0},\mathcal{Y})$ and set $i:=0$. Then proceed as follows:

\smallskip

\noindent {\bf Construction procedure.} If $X_i$ intersects $\mathcal{Y}$, set $\ell:=i$ and stop the construction procedure. Otherwise, use the construction for dimension $d-1$ to construct a facet path $\Gamma'_{X'_iX'_{i+1}}$ in $\Lk(x_i,C)$ from the facet $X'_i:=\Lk(x_i,X_i)$ to the vertex set  $\operatorname{T}_{x_i}^{\mathcal{Y}_i}$. Denote the last facet of the path by $X'_{i+1}$. By forming the join of that path with $x_i$, we obtain a facet path $\Gamma_{X_iX_{i+1}}$ from $X_i$ to the facet $X_{i+1}:=x_i\ast X'_{i+1}$ of $C$. Denote the vertex of ${\operatorname{T}}_{x_i}^{\mathcal{Y}_i}$ it intersects by $x_{i+1}$. Define $\mathcal{Y}_{i+1}:=\pp(x_{i+1},\mathcal{Y}_{i})$. Finally, increase $i$ by one, and repeat from the start.
\smallskip

The concatenation of these facet paths is a \Defn{combinatorial segment} from $X$ to $\mathcal{Y}$.

\medskip

\noindent{\bf Part 2: From any facet $X$ to any other facet $Y$.}

\smallskip 

Using Part 1, construct a facet path from $X$ to the vertex set $\F_0(Y)$ of $X_{\ell+1}:=Y$. If $C$ is of dimension $0$, then this finishes the construction. If $C$ is of dimension $d$ greater than $0$, let $x_\ell$ denote any vertex shared by $X_{\ell}$ and $X_{\ell+1}$, and apply the $(d-1)$-dimensional construction to construct a facet path in $\Lk(x_{\ell},C)$ from $\Lk(x_{\ell},X_{\ell})$ to $\Lk(x_{\ell},X_{\ell+1})$. 

Lift this facet path to a facet path $\Gamma_{X_{\ell}X_{\ell+1}}$ in $C$ by forming the join of the path with $x_{\ell}$. This finishes the construction: The \Defn{combinatorial segment} from $X$ to $Y$ is defined as the concatenation \[\Gamma:=\prod_{i\in (0,\, \cdots,\ell)} \Gamma_{X_{i}X_{i+1}}.\]

\subsection*{The combinatorial segment is non-revisiting}\label{ssc:cnrv}

We start off with some simple observations and notions for combinatorial segments in complexes of dimension $d\ge 1$.

\begin{compactitem}[$\circ$]
\item A combinatorial segment $\Gamma$ comes with a vertex path $(x_0, x_1, \, \cdots , x_{\ell})$. This is a \Defn{shortest vertex path} in $C$, i.e.\ it realizes the distance $\ell = \did(\F_0(X),\mathcal{Y})$ resp.\ $\ell = \did(\F_0(X),\F_0(Y))$. The path $\gamma$ is the \Defn{necklace} of $\Gamma$, the vertices $x_i$, $0\le i\le\ell$, are the \Defn{pearls} of~$\Gamma$.	
\item As in the geometric proof, we denote by $\chi_i$, $0\le i\le \ell+1$, the element in the domain of $\Gamma$ corresponding to $X_i$. Let $a\neq \chi_{\ell+1}$ be any element in the domain of $\Gamma$. If $i$ is chosen so that $a\in [\chi_i,\chi_{i+1}-1]$, then $x_i$ is the \Defn{pearl associated to $a$ in $\gamma$}. By convention, we say that \Defn{$\chi_{\ell+1}$ is associated to the pearl $x_\ell$}.
 \end{compactitem}

\begin{lemma}\label{lem:S}
Assume that $\dim C\ge 1$. If $a$ is an element in the domain of $\Gamma$ associated with pearl $x_i$ such that $i<\ell$, and $v$ is any vertex of $\Gamma(a)$ such that $x_{j}, j>i,$ lies in $\St(v,C)$, then the subpath $\Gamma_{[a,\chi_{j}]}$ lies in $\St(v,C)$. In particular, in this case $X_{j}$ is a facet of $\St(v,C)$ as well.
 \end{lemma}

\begin{proof} 
The lemma is clear if $v$ is in $\gamma$ (i.e.\ if $v$ coincides with $x_i$), and in particular it is clear if $\dim C=1$. To see the case $v\neq x_i$, we use induction on the dimension of $C$: Assume now $\dim C>1$. Clearly, $v$ is connected to both $x_i$ and $x_j$ by edges (since $x_i\in \Gamma(a)\subset \St(v,C)$ and $x_j\subset \St(v,C)$ by assumption. Since the necklace is a shortest path in the $1$-skeleton, we therefore in particular have $j-i\le 2$. We claim that we even have we have $j=i+1$. 

Indeed, if on the other hand $j-i=2$, then $v$, seen as an element in the combinatorial link of $x_i$ in $C$, is a vertex in $\operatorname{T}_{x_i}^{\mathcal{Y}_i}$. Therefore, either $v$ or another vertex of $\Gamma(a)$ coincides with the pearl $x_{i+1}$, which contradicts the assumption that $x_{i}\neq x_{i+1}$ is the pearl associated to $a$.

Now, consider the combinatorial segment $\Gamma':=\Lk(x_i,\Gamma_{[\chi_i,\chi_{i+1}]})$ in $\Lk(x_i,C)$. We argued already that $C$ contains the edges $|vx_i|$ and $|vx_j|$, as well as the edge $|x_ix_j|$ since we now know $x_i$ and $x_j$ to be consecutive pearls. Hence, since $C$ is flag, it contains the triangle on vertices $v$, $x_i$ and $x_j=x_{i+1}$, and it trivially follows that $\St(v,\Lk(x_i,C))$ contains the pearl $x_{i+1}$ of $\Gamma'$. Furthermore, $\Gamma'(a)$ is contained in $\St(v,\Lk(x_i,C))$ since $v\in \Gamma(a)$. Hence, the subpath $\Gamma'_{[a,\chi_{i+1}]}$ of $\Gamma'$ lies in $\St(v,\Lk(x_i,C))$ by induction assumption, so \[\Gamma_{[a,\chi_{i+1}]}=x_i\ast \Gamma'_{[a,\chi_{i+1}]}\subset x_i\ast \St(v,\Lk(x_i,C))\subset\St(v,C).\qedhere\]
\end{proof}

\begin{proof}[\textbf{Second proof of Theorem~\ref{THM:HIRSCHA}}]
Again, we use induction on the dimension; a combinatorial segment is clearly non-revisiting if $\dim C=0$. Assume now $\dim C\geq 1$, and consider a combinatorial segment $\Gamma$ that connects a facet $X$ with a facet $Y$ of $C$, as constructed above. Let $a,\ b$ be in the domain of $\Gamma$, associated with pearls $x_i$ and $x_j$, respectively. Assume that both $\Gamma(a)$ and $\Gamma(b)$ lie in the star of some vertex $v$ of $C$. Then the subpath $\Gamma_{[a,b]}$ of $\Gamma$ lies in the star $\St(v,C)$ of $v$ entirely. To see this, there are two cases to consider:
\begin{compactitem}[$\circ$]
\item {\bf If $i{=}j$:} By the inductive assumption, the combinatorial segment $\Gamma_{[\chi_i,\chi_{i+1}]}$ is non-revisiting, since it was obtained from the combinatorial segment $\Lk(x_i,\Gamma_{[\chi_i,\chi_{i+1}]})$ in the complex $\Lk(x_{i},C)$ by join with~$x_{i}$. Hence, the subpath $\Gamma_{[a,b]}$ of $\Gamma_{[\chi_{i},\chi_{i+1}]}$ is non-revisiting, and consequently lies in $\St(v,C)$.
\item {\bf If $i{<}j$:} 
Since $x_j \in \Gamma(b) \subset St(v,C)$, Lemma~\ref{lem:S} shows that $\Gamma_{[a,\chi_{j}]}$ lies in $\St(v,C)$. Furthermore, we can use the argument of the previous case to show that $\Gamma_{[\chi_{j},b]}$ lies in $\St(v,C)$, so that we obtain \[\Gamma_{[a,b]}=\Gamma_{[a,\chi_{j}]}\cdot \Gamma_{[\chi_{j},b]}\subset \St(v,C).\qedhere\]
\end{compactitem}
\end{proof}

{\small
\bibliographystyle{myamsalpha}
\itemsep=-3.9mm
\bibliography{Ref}
}
  \end{document}